\documentclass[12pt,reqno]{amsart}
\usepackage{amssymb}
\usepackage{amsxtra}
\usepackage{mathrsfs}
\usepackage{paralist}
\usepackage{cite}
\usepackage[unicode,hypertex]{hyperref}
\newtheorem{theorem}{Theorem}[section]
\newtheorem*{theorem*}{Theorem}
\newtheorem{prop}[theorem]{Proposition}
\newtheorem{corollary}[theorem]{Corollary}
\newtheorem{lemma}[theorem]{Lemma}
\theoremstyle{definition}
\newtheorem{definition}[theorem]{Definition}
\newtheorem{example}[theorem]{Example}
\newtheorem{examples}[theorem]{Examples}
\newtheorem{remark}[theorem]{Remark}
\newtheorem*{remark*}{Remark}
\newtheorem*{notation}{Notation}
\newtheorem*{convention}{Convention}
\newtheorem{problem}[theorem]{Problem}
\DeclareMathOperator{\Ker}{Ker}

\DeclareMathOperator{\spn}{span}

\newcommand*{\Ptens}{\mathop{\widehat\otimes}}
\newcommand*{\ptens}[1]{\mathop{\widehat\otimes}_{#1}}

\newcommand*{\xra}{\xrightarrow}
\newcommand*{\h}{\mathbf h}
\newcommand*{\lmod}{\mbox{-}\!\mathop{\mathsf{mod}}}
\newcommand*{\rmod}{\mathop{\mathsf{mod}}\!\mbox{-}}
\newcommand*{\bimod}{\mbox{-}\!\mathop{\mathsf{mod}}\!\mbox{-}}
\newcommand*{\Vect}{\mathsf{Vect}}

\renewcommand*{\dh}{\mathop{\mathrm{dh}}}
\newcommand*{\db}{\mathop{\mathrm{db}}}
\newcommand*{\dg}{\mathop{\mathrm{dg}}}
\newcommand*{\wdh}{\mathop{\mathrm{w.dh}}}
\newcommand*{\wdb}{\mathop{\mathrm{w.db}}}
\newcommand*{\wdg}{\mathop{\mathrm{w.dg}}}
\newcommand*{\CC}{\mathbb C}
\newcommand*{\DD}{\mathbb D}
\newcommand*{\R}{\mathbb R}
\newcommand*{\N}{\mathbb N}
\newcommand*{\Z}{\mathbb Z}
\newcommand*{\cO}{\mathscr O}

\newcommand*{\cH}{\mathscr H}

\newcommand*{\bfU}{\textbf{U}}
\newcommand*{\bfN}{\textbf{N}}
\newcommand*{\bfB}{\textbf{B}}
\newcommand*{\bfM}{\textbf{M}}
\newcommand*{\eps}{\varepsilon}

\begin{document}


\title[Homological properties of K\"othe algebras]{Homological dimensions
and approximate\\ contractibility
for K\"othe algebras}

\author{Alexei Yu. Pirkovskii}
\address{Department of Nonlinear Analysis and Optimization,
Faculty of Science\\
Peoples' Friendship University of Russia\\
Mikluho-Maklaya 6, 117198 Moscow, Russia\\
Email: pirkosha@sci.pfu.edu.ru, pirkosha@online.ru}

\maketitle

\begin{abstract}
We give a survey of our recent results on homological properties
of K\"othe algebras, with an emphasis on biprojectivity, biflatness,
and homological dimension. Some new results on the approximate contractibility
of K\"othe algebras are also presented.
\end{abstract}

\section{Introduction}
From the Banach algebra theory point of view, K\"othe algebras are weighted
locally convex analogues of $\ell^1$ with pointwise multiplication.
The underlying locally convex spaces of K\"othe algebras (K\"othe sequence spaces)
are classical objects, which have been studied since the 1940's
\cite{Kothe_stufen} and are often used to provide various examples and
counterexamples in the theory of topological vector spaces.
The idea to consider K\"othe spaces as algebras under pointwise multiplication
is apparently due to S.~J.~Bhatt and G.~M.~Deheri \cite{Bh_Deh}. The study of homological
properties of K\"othe algebras was started by the author \cite{Pir_bipr}.

The present paper is organized as follows. In Sections 2 and 3, we recall the definitions
and give some basic examples of K\"othe spaces and K\"othe algebras.
In Section 4, we give a brief outline of Topological Homology, i.e., the homology
theory for topological algebras. In Section 5, we introduce and discuss four conditions
(denoted (\bfU), (\bfN), (\bfB), and (\bfM)) on K\"othe sets. These conditions
are then applied to computing homological dimensions of K\"othe algebras in Section 6.
A detailed exposition of the results of Sections 5 and 6 (including proofs) can be found in
\cite{Pir_bipr,Pir_bipr2,Pir_msb,Pir_QJM}.
Section 7 contains some new results on the approximate contractibility of K\"othe
algebras. Among other things, we show that a nuclear biprojective K\"othe algebra
is approximately contractible. Finally, in Section 8 we formulate some open
problems related to homological properties of Fr\'echet algebras
and, in particular, of K\"othe algebras.

\section{K\"othe spaces}

Let $I$ be any set, and let $P$ be a set of nonnegative real-valued functions on $I$.
For $p\in P$, we will write $p_i$ for $p(i)$.
Recall that $P$ is a {\em K\"othe set} on $I$ if the following axioms are
satisfied:
\begin{align*}
\tag*{(P1)}
&\forall\,i\in I\quad\exists\, p\in P:\quad p_i>0\, ;\\
\tag*{(P2)}
&\forall\, p,q\in I\quad\exists\, r\in P:\quad\max\{ p_i,q_i\}\le r_i\quad\forall\, i\in I\, .
\end{align*}
Given a K\"othe set $P$, the {\em K\"othe space}
$\lambda(P)$ is defined as follows:
\begin{equation*}
\lambda(P)=
\Bigl\{ x=(x_i)\in \CC^I :
\| x\|_p=\sum_i |x_i|p_i <\infty\quad\forall\, p\in P\Bigr\}\, .
\end{equation*}
This is a complete locally convex space
with the topology determined by
the family of seminorms $\{\|\cdot\|_p : p\in P\}$. Clearly, $\lambda(P)$
is a Fr\'echet space if and only if $P$ contains an at most countable
cofinal subset.

\begin{convention}
To be definite, we will always assume that $P$ is countable, although some of
our results hold without this assumption.
\end{convention}

Let us give some standard examples of K\"othe spaces. In all these examples,
we set $I=\N$, so that elements of $P$ and of $\lambda(P)$ are sequences.

\begin{example}
\label{example:l1}
If $P$ consists of only one sequence $(1,1,\ldots)$, then $\lambda(P)=\ell^1$.
\end{example}

\begin{example}
For each $n\in\N$ set $p^{(n)}=(1,\ldots ,1,0,\ldots)$ with $1$ repeated $n$ times,
and let $P=\{ p^{(n)} : n\in\N\}$. Then $\lambda(P)=\CC^\N$, the space of all complex
sequences with the topology of pointwise convergence.
More generally, if $I$ is any set and $P$ is the family of all nonnegative
functions on $I$ with finite support, then $\lambda(P)=\CC^I$.
\end{example}

\begin{example}
\label{example:s}
For each $n\in\N$ set $p^{(n)}=(1^n,2^n,\ldots ,k^n,\ldots)$,
and let $P=\{ p^{(n)} : n\in\N\}$. The resulting K\"othe space $\lambda(P)$ is denoted
by $s$ and is called {\em the space of rapidly decreasing sequences}.
\end{example}

\begin{example}
\label{example:O(C)}
For each $n\in\N$ set $p^{(n)}=(n^1,n^2,\ldots ,n^k,\ldots)$,
and let $P=\{ p^{(n)} : n\in\N\}$. It is easy to show that the resulting
K\"othe space $\lambda(P)$ is topologically isomorphic to the space
$\cO(\CC)$ of entire functions with the topology of compact convergence.
Explicitly, the isomorphism $\cO(\CC)\to\lambda(P)$ takes each entire function
to the sequence of its Taylor coefficients at $0$.
\end{example}

\begin{example}
\label{example:pow_ser}
This example generalizes Examples \ref{example:s} and \ref{example:O(C)}.
Fix a real number $0<R\le\infty$ and a nondecreasing sequence
$\alpha=(\alpha_n)_{n\in\N}$ of positive numbers with $\lim_n\alpha_n=\infty$.
Let $P=\{ (r^{\alpha_k})_{k\in\N} : 0<r<R\}$. The resulting K\"othe space $\lambda(P)$
is denoted by $\Lambda_R(\alpha)$ and is called a {\em power series space}
(of \emph{finite type} if $R<\infty$, and of \emph{infinite type} if $R=\infty$).
Of course, the above K\"othe set $P$ is uncountable, but
we can easily make it countable (e.g., by considering only rational $r$)
without changing $\lambda(P)$.

It is easy to see that if $\alpha_n=\log n$, then $\Lambda_\infty(\alpha)=s$
(see Example \ref{example:s}). If $\alpha_n=n$, then
$\Lambda_R(\alpha)$ is topologically isomorphic to $\cO(\DD_R)$, the space
of holomorphic functions on the disk $\DD_R=\{ z\in\CC : |z|<R\}$
(see Example \ref{example:O(C)}).
\end{example}

\section{K\"othe algebras}
Throughout, all vector spaces and algebras are assumed to be over the field $\CC$
of complex numbers. All algebras are assumed to be associative, but
not necessarily unital. The unitization of an algebra $A$ is denoted by $A_+$.
By a {\em Fr\'echet algebra} we mean a complete metrizable locally convex
algebra (i.e., a topological algebra whose underlying space
is a Fr\'echet space). A {\em locally $m$-convex algebra} is a topological
algebra $A$ whose topology can be defined by a family of submultiplicative
seminorms (i.e., seminorms $\|\cdot\|$ satisfying $\| ab\|\le \| a\| \| b\|$
for all $a,b\in A$). Note that, unlike some authors, we do not require
Fr\'echet algebras to be locally $m$-convex.

Given a set $I$, $\CC^I$ is clearly an algebra under pointwise
multiplication. It is natural to ask when $\lambda(P)$ is a subalgebra of $\CC^I$,
and when it is a Fr\'echet algebra (under its canonical topology).
To answer these questions, let us introduce some notation.

\begin{notation}
Let $P$, $Q$ be K\"othe sets on $I$. We say that {\em $P$ is dominated by $Q$}
and write $P\prec Q$ if for each $p\in P$ there exist $q\in Q$ and $C>0$ such that
$p_i\le Cq_i$ for all $i\in I$ (we write $p\le Cq$ for brevity).
This is equivalent to say that $\lambda(Q)\subset\lambda(P)$,
and the embedding of $\lambda(Q)$ into $\lambda(P)$ is continuous.
In fact, since $P$ and $Q$ are assumed to be countable, it is easy to show
that $P\prec Q$ if and only if $\lambda(Q)\subset\lambda(P)$
(the continuity of the embedding will then hold automatically).
If $P\prec Q$ and $Q\prec P$, then we say that
$P$ and $Q$ are {\em equivalent} and write $P\sim Q$.
This means exactly that $\lambda(P)=\lambda(Q)$ topologically, or, equivalently,
that $\lambda(P)=\lambda(Q)$ as sets.
The {\em product} $P\cdot Q$ is defined to be the K\"othe set consisting of
all functions $pq=(p_i q_i)_{i\in I}$, where $p\in P$ and $q\in Q$.
Finally, we let $P^2=P\cdot P$, which is equivalent to the K\"othe set
$\{ p^2=(p_i^2)_{i\in I} : p\in P\}$.
\end{notation}

\begin{prop}
\label{prop:Kothe_alg}
For a K\"othe set $P$, the following conditions are equivalent:
\begin{enumerate}
\item[{\upshape (i)}] $\lambda(P)$ is a subalgebra of $\CC^I$;
\item[{\upshape (ii)}] $\lambda(P)$ is a subalgebra of $\CC^I$ and is a Fr\'echet algebra
under its canonical topology;
\item[{\upshape (iii)}] $P\prec P^2$.
\end{enumerate}
\end{prop}
\begin{proof}
$\mathrm{(i)}\Longrightarrow\mathrm{(iii)}$.
Since $P$ is countable, we may assume that $P=\{ p^{(k)} : k\in\N\}$ and that
$p^{(k)}\le p^{(k+1)}$ for all $k$.
Assume, towards a contradiction, that $P$ is not dominated by $P^2$.
Then there exists $m\in\N$ such that there is no $k\in\N$ and $C>0$ satisfying
$p^{(m)}\le C (p^{(k)})^2$.
Without loss of generality, we may assume that $m=1$.
Then for each $k\in\N$ there exists $i_k\in I$ such that
\begin{equation}
\label{k^4}
p_{i_k}^{(1)} > k^4 (p_{i_k}^{(k)})^2.
\end{equation}
We may also assume that $i_k\ne i_l$ for $k\ne l$.
Note that \eqref{k^4} implies that $p_{i_k}^{(1)}>0$
and hence $p_{i_k}^{(k)}>0$ for all $k$.
Now define a function $x=(x_i)_{i\in I}$ on $I$ by
\begin{equation*}
x_i=
\begin{cases}
\dfrac{1}{k^2 p_{i_k}^{(k)}}, & \text{if $i=i_k$},\\
0, & \text{if $i\notin\{ i_1,i_2,\ldots\}$}.
\end{cases}
\end{equation*}
For each $l\in\N$ we have
\[
\sum_{k\ge l} |x_{i_k}| p_{i_k}^{(l)}
\le \sum_{k\ge l} |x_{i_k}| p_{i_k}^{(k)}
=\sum_{k\ge l} \frac{1}{k^2}<\infty,
\]
so that $x\in\lambda(P)$. On the other hand, \eqref{k^4} implies that
\[
\sum_i |x_i^2| p_i^{(1)}=\sum_k |x_{i_k}^2| p_{i_k}^{(1)}
= \sum_k \frac{p_{i_k}^{(1)}}{k^4 (p_{i_k}^{(k)})^2}=\infty,
\]
i.e., $x^2\notin\lambda(P)$. Thus $\lambda(P)$ is not a subalgebra of $\CC^I$,
which contradicts (i).

$\mathrm{(iii)}\Longrightarrow\mathrm{(ii)}$.
Given $p\in P$, find $q\in P$ and $C>0$ such that $p\le Cq$.
Then for each $a,b\in\lambda(P)$ we have
\begin{equation}
\label{ab}
\| ab\|_p=\sum_i |a_i b_i| p_i\le C\sum_i | a_i| |b_i| q_i^2
\le C \| a\|_q \| b\|_q<\infty.
\end{equation}
Therefore $ab\in\lambda(P)$, i.e., $\lambda(P)$ is a subalgebra of $\CC^I$.
The continuity of the multiplication on $\lambda(P)$ is immediate from \eqref{ab}.

$\mathrm{(ii)}\Longrightarrow\mathrm{(i)}$. This is clear.
\end{proof}

\begin{remark}
If $P$ is arbitrary (i.e., not necessarily countable), then condition (iii)
is equivalent to the assertion that $\lambda(P)$
is a subalgebra of $\CC^I$ and is a topological algebra with jointly continuous
multiplication.
\end{remark}

\begin{definition}
Algebras of the form $\lambda(P)$ (where $P$ is a K\"othe set
satisfying $P\prec P^2$) are called {\em K\"othe algebras}.
\end{definition}

\begin{remark}
In many natural cases (see examples below), we have $p_i\in\{ 0\}\cup [1,+\infty)$ for
all $p\in P,\; i\in I$. This implies that $p_i\le p_i^2$, so that $P\prec P^2$,
and $\lambda(P)$ is a K\"othe algebra. Moreover, the seminorms $\|\cdot\|_p$
are submultiplicative in this case, and so $\lambda(P)$ is locally $m$-convex.
\end{remark}

\begin{examples}
Clearly, $\ell^1$ and $\CC^I$ are K\"othe algebras.
The power series space $\Lambda_R(\alpha)$ is a K\"othe algebra if and only
if $R\ge 1$. Moreover, if $R>1$, then $\Lambda_R(\alpha)$ is locally $m$-convex.
In particular, identifying the space $\cO(\DD_R)$ (for $R\ge 1$)
with $\Lambda_R(\{ n\})$ (see Example \ref{example:pow_ser}), we see that
$\cO(\DD_R)$ becomes a Fr\'echet algebra under the
``componentwise'' product of the Taylor expansions
of holomorphic functions (the {\em Hadamard product}; see \cite{Render}).
The resulting topological algebra is denoted by $\cH(\mathbb D_R)$
and is called the {\em Hadamard algebra}.
\end{examples}

\section{Topological Homology}
Now let us briefly recall some basic notions of Topological Homology, i.e.,
the homology theory for locally convex topological algebras. This theory was developed
in the early 1970's by A.~Ya.~Helemskii (see, e.g., \cite{X_dg}) in the special case
of Banach algebras.
A few years later a similar theory was independently discovered
by R.~Kiehl and J.~L.~Verdier \cite{KV} and by J.~L.~Taylor \cite{T1}
in the context of more general topological algebras.
Let us briefly recall the basics of this theory.
For details, we refer to Helemskii's monograph \cite{X1}.

To be definite, we will work only with Fr\'echet modules over Fr\'echet algebras.
Let $A$ be a Fr\'echet algebra.
A {\em left Fr\'echet $A$-module} is a left $A$-module $X$ endowed with
a Fr\'echet space topology in such a way that the action $A\times X\to X$
is continuous.
Left Fr\'echet $A$-modules and their continuous morphisms form a category
denoted by $A\lmod$. Given $X,Y\in A\lmod$,
the space of morphisms from $X$ to $Y$ will be denoted by $\h_A(X,Y)$.
The categories $\rmod A$ and $A\bimod A$ of right
Fr\'echet $A$-modules and of Fr\'echet $A$-bimodules are defined similarly.

If $X$ is a right Fr\'echet $A$-module and $Y$
is a left Fr\'echet $A$-module, then their {\em $A$-module tensor product}
$X\ptens{A}Y$ is defined to be
the quotient $(X\Ptens Y)/N$, where $N\subset X\Ptens Y$
is the closed linear span of all elements of the form
$x\cdot a\otimes y-x\otimes a\cdot y$
($x\in X$, $y\in Y$, $a\in A$).
As in pure algebra, the $A$-module tensor product can be characterized
by the universal property that, for each Fr\'echet space $E$,
there is a natural bijection between the set of all
continuous $A$-balanced bilinear maps from $X\times Y$ to $E$
and the set of all continuous linear maps from
$X\ptens{A}Y$ to $E$.

A chain complex $C_\bullet=(C_n,d_n)_{n\in\Z}$ in $A\lmod$ is {\em admissible} if
it splits in the category of topological vector spaces, i.e., if it has
a contracting homotopy consisting of continuous linear maps. Geometrically,
this means that $C_\bullet$ is exact, and $\Ker d_n$ is a complemented subspace
of $C_n$ for each $n$.

Let $\Vect$ denote the category of vector spaces and linear maps.
A left Fr\'echet $A$-module $P$ is {\em projective} if the functor
$\h_A(P,-)\colon A\lmod\to\Vect$ is exact in the sense that it takes admissible sequences
of Fr\'echet $A$-modules to exact sequences of vector spaces.
A left Fr\'echet $A$-module $F$ is {\em flat} if the tensor product functor
$(-)\ptens{A} F\colon \rmod A\to\Vect$ is exact in the same sense as above.
It is known that every projective Fr\'echet module is flat.

A {\em resolution} of $X\in A\lmod$ is a pair $(P_\bullet,\eps)$
consisting of a nonnegative chain complex
$P_\bullet$ in $A\lmod$ and a morphism $\eps\colon P_0\to X$ making the sequence
$P_\bullet\xra{\eps} X\to 0$ into an admissible complex.
The {\em length} of $P_\bullet$ is the minimum integer $n$
such that $P_i=0$ for all $i>n$, or $\infty$ if there is no such $n$.
If all the $P_i$'s are projective (respectively, flat), then
$(P_\bullet,\eps)$ is called a {\em projective resolution}
(respectively, a {\em flat resolution}) of $X$.
It is a standard fact that $A\lmod$ has {\em enough projectives},
i.e., each left Fr\'echet $A$-module has a projective resolution.
The same is true of $\rmod A$ and $A\bimod A$.

The {\em projective homological dimension} of $X\in A\lmod$ is the
minimum integer $n=\dh_A X\in\Z_+\cup\{\infty\}$
with the property that $X$ has a projective resolution of length $n$.
Similarly, the {\em weak homological dimension} of $X\in A\lmod$ is the
minimum integer $n=\wdh_A X\in\Z_+\cup\{\infty\}$
with the property that $X$ has a flat resolution of length $n$.
Heuristically, the projective (respectively, flat) dimension measures
how far $X$ is from being projective (respectively, flat).
In particular, $\dh_A X=0$ if and only if $X$ is projective, and
$\wdh_A X=0$ if and only if $X$ is flat.
Since each projective module is flat, we clearly have $\wdh_A X\le\dh_A X$.

Given a Fr\'echet algebra $A$, the {\em global dimension} and
the {\em weak global dimension} of $A$
are defined by
\begin{align*}
\dg A&=\sup\{ \dh\nolimits_A X \,|\, X\in A\lmod\},\\
\wdg A&=\sup\{ \wdh\nolimits_A X \,|\, X\in A\lmod\}.
\end{align*}
The {\em bidimension} and the {\em weak bidimension} of
$A$ are defined by
\begin{align*}
\db A&=\dh\nolimits_{A-A} A_+,\\
\wdb A&=\wdh\nolimits_{A-A} A_+,
\end{align*}
where the subscript ``$A-A$'' means that we work in the category $A\bimod A$.
We clearly have $\wdg A\le\dg A$ and $\wdb A\le\db A$.
It is also true (but less obvious) that $\dg A\le\db A$ and $\wdg A\le\wdb A$.

To complete our short introduction to Topological Homology, let us recall some
``homological triviality'' conditions for Fr\'echet algebras.
By definition, a Fr\'echet algebra $A$ is {\em biprojective}
(respectively, {\em biflat}) if $A$ is projective (respectively, flat)
in $A\bimod A$. Similarly, $A$ is {\em contractible}
(respectively, {\em amenable}) if $A_+$ is projective (respectively, flat)
in $A\bimod A$.
In fact, there are many other equivalent definitions of the above notions.
For example, it is immediate from the definition of $\db$ that
$A$ is contractible if and only if $\db A=0$. This is equivalent to say that
for each Fr\'echet $A$-bimodule $X$ all continuous derivations
from $A$ to $X$ are inner (see \cite[7.3.37]{X2}). Similarly, $A$ is amenable
if and only if $\wdb A=0$,
which is equivalent to say that for each Fr\'echet $A$-bimodule $X$ all continuous
derivations from $A$ to $X^*$ (where $X^*$ is the strong dual of $X$)
are inner (see \cite[7.3.37]{X2} for the Banach algebra case and
\cite[Corollary 3.5]{Pir_msb} for the Fr\'echet algebra case).
In fact, the above characterization of amenability of Banach algebras
in terms of derivations was the original definition
of amenability due to B.~E.~Johnson \cite{Johnson_CBA}.

Let us also remark that a Fr\'echet algebra $A$ is biprojective
if and only if the product map $A\Ptens A\to A$ is a retraction
in $A\bimod A$ \cite[IV.5.6]{X1}. There is also a dual characterization
of biflat Banach algebras \cite[VII.2.7]{X1}, but we will not use it.
Recall also that a Fr\'echet algebra $A$ is contractible if
and only if $A$ is biprojective and unital \cite[IV.5.8]{X1}.

The following theorem is essentially due to Helemskii \cite[V.2.28]{X1},
\cite[2.5.8]{X_HOA}; see also \cite[Proposition 4.8]{Pir_msb} for
the Fr\'echet algebra case.

\begin{theorem*}
Let $A$ be a Fr\'echet algebra.
\begin{compactenum}
\item[{\upshape (i)}] If $A$ is biprojective, then $\db A\le 2$.
\item[{\upshape (ii)}] If $A$ is biflat, then $\wdg A\le 2$.
\item[{\upshape (iii)}] If $A$ is biflat and projective in $A\lmod$
or in $\rmod A$, then $\wdb A\le 2$.
\item[{\upshape (iv)}] If $A$ is a biflat Banach algebra, then $\wdb A\le 2$.
\end{compactenum}
\end{theorem*}

We do not know whether part (iv) of the above theorem holds for
biflat Fr\'echet algebras.

\section{Some conditions on K\"othe sets}
Our task is to compute the homological dimensions $\dg$, $\db$, $\wdg$,
and $\wdb$ of K\"othe algebras. To this end, let us introduce some conditions
on K\"othe sets.

Let $\lambda(P)$ be a K\"othe algebra.

\begin{definition}
We say that $P$ satisfies (\bfU) if
\begin{equation*}
\sum_i p_i<\infty\quad\text{for all $p\in P$}.
\end{equation*}
\end{definition}

Clearly, $P$ satisfies (\bfU) if and only if $\lambda(P)$ is unital
(whence the notation ``(\bfU)'').

\begin{example}
$\CC^I$ is unital.
\end{example}

\begin{example}
The Hadamard algebra $\cH(\DD_1)$ is unital. The identity element of $\cH(\DD_1)$
is the function $f(z)=(1-z)^{-1}$, whose Taylor coefficients are all equal to $1$.
\end{example}

\begin{example}
\label{example:Lambda_notuni}
The algebra $\Lambda_R(\alpha)$ is not unital unless $R=1$. In particular,
the algebra $s$ of rapidly decreasing sequences and the Hadamard algebras
$\cH(\DD_R)$ for $R>1$ are not unital.
\end{example}

\begin{definition}
\label{def:N}
We say that $P$ satisfies (\bfN) if
\begin{equation*}
\forall p\in P\quad \exists\, q\in P\quad\exists\,\alpha\in\ell^1(I)\; :\quad
p\le\alpha q.
\end{equation*}
\end{definition}

The {\em Grothendieck--Pietsch criterion} (see \cite[6.1.2]{Pietsch}) states that
$P$ satisfies (\bfN) if and only if $\lambda(P)$ is nuclear
(whence the notation ``(\bfN)'').

There is another useful form of the Grothendieck--Pietsch criterion. Given a K\"othe set $P$,
let
\begin{equation*}
\lambda^\infty(P)=
\Bigl\{ a=(a_i)\in \CC^I :
\| a\|_p^\infty=\sup_i |a_i|p_i <\infty\quad\forall\, p\in P\Bigr\}\, .
\end{equation*}
This is a complete locally convex space
with the topology determined by
the family of seminorms $\{\|\cdot\|_p^\infty : p\in P\}$.
Clearly, $\lambda(P)\subset\lambda^\infty(P)$, and the embedding is continuous.
The Grothendieck-Pietsch criterion can be reformulated as follows:
$\lambda(P)$ is nuclear if and only if
$\lambda(P)=\lambda^\infty(P)$ topologically (see \cite[6.1.3]{Pietsch}).
Since $P$ is assumed to be countable, this is equivalent to say that
$\lambda(P)=\lambda^\infty(P)$ as sets (see \cite[28.16]{MV}).

\begin{example}
$\CC^I$ is nuclear. This is a standard fact (moreover, it is known that
a product of nuclear spaces is nuclear), but this is also immediate from the
Grothendieck--Pietsch criterion.
\end{example}

\begin{example}
\label{example:Lambda_R_nucl}
The Grothendieck--Pietsch criterion implies that
\[
\Lambda_R(\alpha)\text{ is nuclear } \iff
\begin{cases}
\sup_n (\log n)/\alpha_n<\infty & \text{ if } R=\infty;\\
\lim_n (\log n)/\alpha_n=0 & \text{ if } R<\infty
\end{cases}
\]
(see, e.g., \cite[29.6 and 28.16]{MV}).
In particular, $s$ and $\cH(\DD_R)$ are nuclear.
\end{example}

\begin{definition}
We say that $P$ satisfies (\bfB) if
\begin{equation*}
P\sim P^2.
\end{equation*}
\end{definition}

The next theorem justifies the notation ``(\bfB)''.

\begin{theorem}[{\cite[5.2]{Pir_msb}}]
\label{thm:kothe_bipr}
Let $A=\lambda(P)$ be a K\"othe algebra. The following conditions
are equivalent:
\begin{compactenum}
\item[{\upshape (i)}] $P$ satisfies {\upshape (\bfB)};
\item[{\upshape (ii)}] $A$ is biprojective;
\item[{\upshape (iii)}] $A$ is biflat;
\item[{\upshape (iv)}] $A$ is flat in $A\lmod$;
\item[{\upshape (v)}] the product map $A\Ptens A\to A,\; a\otimes b\mapsto ab$, is onto.
\end{compactenum}
\end{theorem}

\begin{example}
\label{example:CC^I_contr}
$\CC^I$ is biprojective. Since $\CC^I$ is also unital, it follows that
$\CC^I$ is contractible. This was first observed by Taylor \cite[5.9]{T1}
(see also \cite[IV.5.27]{X1}),
but this is also immediate from Theorem~\ref{thm:kothe_bipr}.
\end{example}

\begin{remark}
As was shown by Helemskii \cite[IV.5.27]{X1}, each complete commutative
contractible locally $m$-convex
algebra is isomorphic to $\CC^I$ for some $I$.
\end{remark}

\begin{example}
\label{example:l1_bipr}
$\ell^1$ is biprojective. This fact is due to Helemskii \cite[IV.5.9]{X1},
but this is also immediate from Theorem~\ref{thm:kothe_bipr}.
\end{example}

\begin{example}
\label{example:Lambda_R_bipr}
Theorem \ref{thm:kothe_bipr} implies that $\Lambda_R(\alpha)$ is biprojective
if and only if either $R=1$ or $R=\infty$
\cite[Example 3.5]{Pir_bipr}.
For instance, $s$, $\cH(\CC)$, $\cH(\DD_1)$ are biprojective.
Moreover, $\cH(\DD_1)$ is unital and hence contractible.
On the other hand, $\cH(\DD_R)$ is not biprojective unless $R=1$ or $R=\infty$.
\end{example}

\begin{definition}
We say that $P$ satisfies (\bfM) if
there exist complex matrices $\alpha=(\alpha_{ij})_{i,j\in I}$
and $\beta=(\beta_{ij})_{i,j\in I}$ such that
\begin{compactenum}
\item[\quad (\bfM1)]
$\alpha_{ij}+\beta_{ij}=1\quad (i,j\in I)$;
\item[\quad (\bfM2)]
$\forall p\in P\quad\exists\, C>0\quad\exists\, q\in P\quad\forall j\in\N\quad
\sup_i |\alpha_{ij}|p_i p_j\le Cq_j^2$;
\item[\quad (\bfM3)]
$\forall p\in P\quad\exists\, C>0\quad\exists\, q\in P\quad\forall i\in\N\quad
\sup_j |\beta_{ij}|p_j p_i\le Cq_i^2$.
\end{compactenum}
\end{definition}

The importance of condition (\bfM) will become clear in Section~\ref{sect:dims}.
Since this condition looks a bit artificial and difficult to check, let show that
it is often satisfied automatically.

\begin{prop}
\label{prop:M_aut}
Let $P$ be a K\"othe set on $\N$, and suppose that
either each $p\in P$ is nondecreasing, or each $p\in P$
is nonincreasing. Then $P$ satisfies {\upshape (\bfM)}.
\end{prop}
\begin{proof}
For convenience, let us agree that $a/0=+\infty$ and
$a+(+\infty)=+\infty$ for all $a\ge 0$.
The advantage of this convention is that
\begin{equation}
\label{infty}
a\ge bc \iff a/b\ge c \quad\text{for all $a,b,c\ge 0$}.
\end{equation}
For each $i,j\in \N$ set
\[
\alpha_{ij}=\min\left\{ 1,\; \inf_{p\in P} \frac{p_j}{p_i}\right\}.
\]
By \eqref{infty}, we have
\[
\alpha_{ij} p_i p_j\le p_j^2\qquad (p\in P,\; i,j\in \N),
\]
and so $\alpha=(\alpha_{ij})_{i,j\in \N}$ satisfies (\bfM2).
Now observe that the condition imposed on $P$ implies the inequality
\[
\frac{p_i}{p_j}+\frac{q_j}{q_i}\ge 1\qquad (p,q\in P,\; i,j\in \N),
\]
whence
\begin{equation}
\label{alpha}
\frac{p_i}{p_j}+\alpha_{ij}\ge 1 \qquad (p\in P,\; i,j\in \N).
\end{equation}
Setting $\beta_{ij}=1-\alpha_{ij}$ and using \eqref{alpha} and \eqref{infty}, we see that
\[
\beta_{ij} p_i p_j\le p_i^2\qquad (p\in P,\; i,j\in \N),
\]
and so $\beta=(\beta_{ij})_{i,j\in \N}$ satisfies (\bfM3).
Clearly, (\bfM1) is also satisfied.
\end{proof}

\begin{example}
\label{example:M}
The K\"othe sets from Examples \ref{example:l1}--\ref{example:pow_ser}
satisfy the conditions of Proposition~\ref{prop:M_aut},
and hence they satisfy condition (\bfM).
Strictly speaking, this is not the case for $\Lambda_R(\alpha)$ if $R>1$,
but in this case the K\"othe set
$P=\{ (r^{\alpha_k})_{k\in\N} : 0<r<R\}$ is equivalent to the K\"othe set
$P'=\{ (r^{\alpha_k})_{k\in\N} : 1<r<R\}$, which satisfies the
conditions of Proposition~\ref{prop:M_aut}.
\end{example}

\begin{remark}
Condition (\bfM) slightly differs from its original version
in \cite[4.7]{Pir_bipr} (see also \cite[7.2]{Pir_msb}, \cite[Section 3]{Pir_QJM}).
Namely, in \cite{Pir_bipr,Pir_msb,Pir_QJM} the right-hand side of
(\bfM2) (respectively, (\bfM3)) was $Cq_j$ (respectively, $Cq_i$).
We have made this change in order to prove Proposition~\ref{prop:M_aut}.
This will not affect our main results (Theorems~\ref{thm:wdgwdb}
and \ref{thm:dgdb}), because for K\"othe sets satisfying (\bfB) the old
form of (\bfM) is clearly equivalent to the new one.
\end{remark}

\section{Homological dimensions of K\"othe algebras}
\label{sect:dims}
In this section we compute the dimensions $\dg$, $\db$, $\wdg$, and $\wdb$ of K\"othe
algebras in terms of conditions (\bfU), (\bfN), (\bfB), (\bfM) introduced above.
Before formulating our results, let us observe that, if $\lambda(P)$ is a K\"othe algebra
(i.e., if $P\prec P^2$),
then $\lambda^\infty(P)$ is also a Fr\'echet algebra under
pointwise multiplication (see the proof of Proposition~\ref{prop:Kothe_alg},
$\mathrm{(iii)}\Longrightarrow\mathrm{(ii)}$).
Since $\lambda(P)$ in continuously embedded into $\lambda^\infty(P)$,
we may consider $\lambda^\infty(P)$ as a left Fr\'ehcet $\lambda(P)$-module.

Given a Fr\'echet algebra $A$, we consider $\CC$ as a left Fr\'echet $A$-module by letting
$A$ act on $\CC$ trivially. In other words, $\CC=A_+/A$.

\begin{theorem}[{\cite[4.3]{Pir_QJM}}]
\label{thm:wdgwdb}
Let $A=\lambda(P)$ be a K\"othe algebra. Then
\begin{equation*}
\wdg A=\wdb A=
\begin{cases}
0, & \text{$P$ satisfies {\upshape (\bfU)}}.\\
1, & \parbox[t]{85mm}{$P$ satisfies {\upshape (\bfB)} and {\upshape (\bfN)},
but does not satisfy {\upshape (\bfU)}. In this case, $\wdh_A\CC=1.$}\\
2, & \parbox[t]{85mm}{$P$ satisfies {\upshape (\bfB)},
but does not satisfy {\upshape (\bfN)}. In this case, $\wdh_A\lambda^\infty(P)=2$.}\\
\infty, & \parbox[t]{85mm}{$P$ does not satisfy {\upshape (\bfB)}. In this case,
$\wdh_A\CC=\infty$.}
\end{cases}
\end{equation*}
\end{theorem}

Let us note that Theorem~\ref{thm:wdgwdb} involves only conditions
(\bfU), (\bfB), and (\bfN), but not (\bfM).
The situation with the ``strong'' dimensions $\dg$ and $\db$ is slightly
more delicate. To compute them, we need some extra notation.

Let $P$ be a K\"othe set on $I$. For each $p\in P$ we define a function
$\bar p\colon I\to\R_+$
by $\bar p_i=\min\{ p_i,1\}$. Clearly, $\bar P=\{ \bar p : p\in P\}$ is a K\"othe set.
It is easy to show (see \cite[3.2]{Pir_QJM}) that if $\lambda(P)$ is a K\"othe
algebra, then the pointwise product $a\cdot x$ of any $a\in\lambda(P)$
and $x\in\lambda(\bar P)$ is in $\lambda(\bar P)$, and that
$\lambda(\bar P)$ becomes a Fr\'echet $\lambda(P)$-module under this operation.

\begin{theorem}[{\cite[4.3]{Pir_QJM}}]
\label{thm:dgdb}
Let $A=\lambda(P)$ be a K\"othe algebra. Then
\begin{equation*}
\dg A=\db A=
\begin{cases}
0, & \text{$P$ satisfies {\upshape (\bfU)}}.\\
1, & \parbox[t]{85mm}{$P$ satisfies {\upshape (\bfB)}, {\upshape (\bfN)},
and {\upshape (\bfM)}, but does not satisfy {\upshape (\bfU)}. In this case, $\dh_A\CC=1.$}\\
2, & \parbox[t]{85mm}{$P$ satisfies {\upshape (\bfB)} and {\upshape (\bfN)},
but does not satisfy {\upshape (\bfM)}. In this case, $\dh_A\lambda(\bar P)=2$.}\\
2, & \parbox[t]{85mm}{$P$ satisfies {\upshape (\bfB)},
but does not satisfy {\upshape (\bfN)}. In this case, $\dh_A\lambda^\infty(P)=2$.}\\
\infty, & \parbox[t]{85mm}{$P$ does not satisfy {\upshape (\bfB)}. In this case,
$\dh_A\CC=\infty$.}
\end{cases}
\end{equation*}
\end{theorem}

\begin{corollary}
\label{cor:unital}
For a K\"othe algebra $\lambda(P)$, the following conditions are equivalent:
\begin{compactenum}
\item[{\upshape (i)}] $\lambda(P)$ is amenable;
\item[{\upshape (ii)}] $\lambda(P)$ is contractible;
\item[{\upshape (iii)}] $\lambda(P)$ is unital.
\end{compactenum}
\end{corollary}

\begin{corollary}
Condition {\upshape (\bfU)} implies conditions {\upshape (\bfB), (\bfN), (\bfM)}.
\end{corollary}

Let us look at some examples. We will see, in particular, that every combination
of (\bfU), (\bfN), (\bfB), (\bfM) mentioned in Theorems~\ref{thm:wdgwdb}
and \ref{thm:dgdb} is possible.

\begin{example}
The algebra $\CC^I$ satisfies {\upshape (\bfU)}, and so
\[
\dg\CC^I=\db\CC^I=\wdg\CC^I=\wdb\CC^I=0
\]
(see Example \ref{example:CC^I_contr} or Corollary \ref{cor:unital}).
\end{example}

\begin{example}
The algebra $\ell^1$ satisfies {\upshape (\bfB)},
but does not satisfy {\upshape (\bfN)}. Therefore Theorems \ref{thm:wdgwdb} and \ref{thm:dgdb}
imply that
\[
\dg\ell^1=\db\ell^1=\wdg\ell^1=\wdb\ell^1=\dh\nolimits_{\ell^1}\ell^\infty
=\wdh\nolimits_{\ell^1}\ell^\infty=2.
\]
For $\dg$, $\db$, and $\dh_{\ell^1}\ell^\infty$, this is an old
result by Helemskii \cite{X_ne1} (see also \cite[V.2.16]{X1});
for $\wdb$, the result is due to Yu.~V.~Selivanov \cite{Sel_bifl}.
\end{example}

\begin{example}
The algebra $\Lambda_\infty(\alpha)$ does not satisfy (\bfU)
(see Example~\ref{example:Lambda_notuni}),
but satisfies (\bfB) and (\bfM)
(see Examples~\ref{example:Lambda_R_bipr} and~\ref{example:M}).
Taking into account Example~\ref{example:Lambda_R_nucl}
and Theorems~\ref{thm:wdgwdb}
and~\ref{thm:dgdb}, we get
\[
\dg\Lambda_\infty(\alpha)=\db\Lambda_\infty(\alpha)
=\wdg\Lambda_\infty(\alpha)=\wdb\Lambda_\infty(\alpha)=
\begin{cases}
1, &  \sup_n (\log n)/\alpha_n<\infty,\\
2 & \text{otherwise.}
\end{cases}
\]
In particular,
\begin{align*}
\dg s&=\db s=\wdg s=\wdb s=1,\\
\dg \cH(\CC)&=\db \cH(\CC)=\wdg \cH(\CC)=\wdb \cH(\CC)=1.
\end{align*}
\end{example}

\begin{example}
The algebra $\Lambda_1(\alpha)$ satisfies {\upshape (\bfB)}
(see Example~\ref{example:Lambda_R_bipr}). By using the nuclearity criterion
for $\Lambda_1(\alpha)$ (see Example~\ref{example:Lambda_R_nucl}),
it is easy to show that if $\Lambda_1(\alpha)$ satisfies
{\upshape (\bfN)}, then it satisfies (\bfU).
Taking into account Theorems~\ref{thm:wdgwdb}
and~\ref{thm:dgdb}, we see that
\[
\dg\Lambda_1(\alpha)=\db\Lambda_1(\alpha)
=\wdg\Lambda_1(\alpha)=\wdb\Lambda_1(\alpha)=
\begin{cases}
0, &  \lim_n (\log n)/\alpha_n<\infty,\\
2 & \text{otherwise.}
\end{cases}
\]
In particular,
\[
\dg \cH(\DD_1)=\db \cH(\DD_1)=\wdg \cH(\DD_1)=\wdb \cH(\DD_1)=0.
\]
\end{example}

\begin{example}
If $1<R<+\infty$, then $\Lambda_R(\alpha)$ does not satisfy (\bfB)
(see Example~\ref{example:Lambda_R_bipr}). Therefore Theorems~\ref{thm:wdgwdb}
and~\ref{thm:dgdb} imply that
\[
\dg\Lambda_R(\alpha)=\db\Lambda_R(\alpha)
=\wdg\Lambda_R(\alpha)=\wdb\Lambda_R(\alpha)=\infty.
\]
In particular,
\[
\dg \cH(\DD_R)=\db \cH(\DD_R)=\wdg \cH(\DD_R)=\wdb \cH(\DD_R)=\infty.
\]
\end{example}

\begin{example}
Let $I=\N\times\N$. For each $i,j,k\in\N$ we define
\[
p_{ij}^{(k)}=
\begin{cases}
2^{(kj)^i} (i+j)^k & (i\le k),\\
(i+j)^k & (i>k).
\end{cases}
\]
Set $p^{(k)}=(p_{ij}^{(k)})_{i,j\in\N}$, and consider the K\"othe set
$P=\{ p^{(k)}\}_{k\in\N}$. As was shown in \cite[Theorem 7.9]{Pir_msb},
$P$ satisfies {\upshape (\bfB)} and {\upshape (\bfN)},
but does not satisfy {\upshape (\bfM)}.
Note that, since $p_{ij}^{(k)}\ge 1$ for all $i,j,k$, we have $\lambda(\bar P)=\ell^1$.
Now Theorems~\ref{thm:wdgwdb}
and~\ref{thm:dgdb} imply that
\begin{align*}
\dg\lambda(P)&=\db\lambda(P)=\dh\nolimits_{\lambda(P)}\ell^1=2, \quad\text{while}\\
\wdg\lambda(P)&=\wdb\lambda(P)=\wdh\nolimits_{\lambda(P)}\CC=1.
\end{align*}
\end{example}

\section{Approximate contractibility of K\"othe algebras and around}
Approximately contractible and approximately amenable Banach algebras were introduced
by F.~Ghahramani and R.~J.~Loy \cite{GL_genamen} and have attracted much attention
in recent years \cite{DLZ,GLZ_genamen2,GS,CGZ,CG}. Similar notions for Fr\'echet algebras were
studied by P.~Lawson and C.~J.~Read~\cite{LR}. By definition, a Fr\'echet algebra
$A$ is {\em approximately contractible} (respectively, {\em approximately amenable})
if for each Fr\'echet $A$-bimodule $X$ every continuous derivation from
$A$ to $X$ (respectively, to the strong dual $X^*$) is the limit of a pointwise
convergent net of inner derivations. In \cite{GLZ_genamen2}, it was shown that
approximate contractibility and approximate amenability for Banach algebras
are equivalent. A similar result for locally $m$-convex Fr\'echet algebras
was proved in~\cite{LR}. We refer the reader to Y.~Zhang's survey~\cite{Zhang_slides}
for a detailed discussion and numerous references concerning
approximate amenability for Banach algebras.

In \cite{LR}, Lawson and Read gave some sufficient conditions of approximate
contractibility for locally $m$-convex K\"othe algebras. In particular, they
showed that the algebra $s$ of rapidly decreasing sequences is approximately
contractible. The following result generalizes Proposition 3.9 from \cite{LR}.

\begin{theorem}
\label{thm:BN_appr}
Let $A=\lambda(P)$ be a K\"othe algebra satisfying {\upshape (\bfB)} and
{\upshape (\bfN)}. Then $A$ is approximately contractible.
\end{theorem}

Before proving Theorem~\ref{thm:BN_appr}, let us recall a result from \cite{LR}.

\begin{lemma}[{\cite[Lemma 2.7]{LR}}]
\label{lemma:LR}
Let $A$ be a Fr\'echet algebra, and let $\{ \|\cdot\|_p : p\in P\}$ be a family
of seminorms generating the topology of $A$. Suppose that there exist nets
$(u_\lambda)_{\lambda\in\Lambda}$ in $A$ and $(d_\lambda)_{\lambda\in\Lambda}$
in $A\Ptens A$ such that
\begin{compactenum}
\item[{\upshape (i)}] $(u_\lambda)$ is a two-sided approximate identity in $A$;
\item[{\upshape (ii)}] $\| a-au_\lambda\|_p \| u_\lambda\|_p\to 0$
and $\| a-u_\lambda a\|_p \| u_\lambda\|_p\to 0$ \qquad $(a\in A)$;
\item[{\upshape (iii)}] $\pi(d_\lambda)=2u_\lambda-u_\lambda^2$, where
$\pi\colon A\Ptens A\to A$ is the product map;
\item[{\upshape (iv)}] $a\cdot d_\lambda-d_\lambda\cdot a\to 0$ \qquad $(a\in A)$.
\end{compactenum}
Then $A$ is approximately contractible.
\end{lemma}

\begin{remark}
In \cite{LR}, the authors consider only locally $m$-convex Fr\'echet algebras.
However, the argument used in \cite{LR} shows that both Lemma~\ref{lemma:LR}
and \cite[Theorem 2.4]{LR}, on which Lemma~\ref{lemma:LR} is based,
are true for all Fr\'echet algebras, not necessarily locally $m$-convex.
\end{remark}

In the case of K\"othe algebras, Lemma~\ref{lemma:LR} can be simplified as follows.

\begin{lemma}
\label{lemma:a-au}
Let $A=\lambda(P)$ be a K\"othe algebra, and let
$\Pi\subset A$ denote the set of characteristic functions $\chi_J$,
where $J$ runs over the family of all finite subsets of $I$.
Suppose that for each $p\in P$ and each $a\in A$ there exists
$u\in\Pi$ such that
\begin{gather}
\label{a-au}
\| a-au\|_p<1,\\
\label{a-au_u}
\| a-au\|_p \| u\|_p < 1.
\end{gather}
Then $A$ is approximately contractible.
\end{lemma}
\begin{proof}
Let $\mathscr F$ denote the family of all finite subsets of $A$, and let
\[
\Lambda=\{ (F,p,\eps) : F\in\mathscr F,\; p\in P,\; \eps>0\}.
\]
We make $\Lambda$ into a directed poset by setting
\[
(F,p,\eps)\preceq (F',p',\eps') \quad\iff\quad F\subset F',\; p\le p',\; \eps\ge\eps'.
\]
Fix any $\lambda=(F,p,\eps)\in\Lambda$, and define $b\in A$ by
\[
b=\eps^{-1}\sum_{a\in F} |a|.
\]
Using \eqref{a-au} and \eqref{a-au_u}, find $u_\lambda=\chi_J\in\Pi$ such that
\begin{gather*}
\| b-bu_\lambda\|_p<1,\\
\| b-bu_\lambda\|_p \| u_\lambda\|_p < 1.
\end{gather*}
Since
\begin{equation*}
\| a-au\|_p=\sum_{i\in I\setminus J} |a_i| p_i,
\end{equation*}
it follows that
\begin{gather*}
\| a-au_\lambda\|_p<\eps \qquad (a\in F),\\
\| a-au_\lambda\|_p \| u_\lambda\|_p < \eps \qquad (a\in F).
\end{gather*}
Therefore the net $(u_\lambda)_{\lambda\in\Lambda}$ satisfies
conditions (i) and (ii) of Lemma~\ref{lemma:LR}.
Now, following \cite[Example 3.1]{LR}, define $d_\lambda\in A\Ptens A$ by
\[
d_\lambda=\sum_{i\in F} e_i\otimes e_i,
\]
where $e_i$ stands for the function on $I$ which is $1$ at $i$, $0$ elsewhere.
Then it is clear that
\begin{gather*}
\pi(d_\lambda)=u_\lambda=2u_\lambda-u_\lambda^2;\\
a\cdot d_\lambda=d_\lambda\cdot a \qquad (a\in A).
\end{gather*}
Thus conditions (iii) and (iv) of Lemma~\ref{lemma:LR} are also satisfied,
and so $A$ is approximately contractible.
\end{proof}

\begin{proof}[Proof of Theorem~\ref{thm:BN_appr}]
Fix any $a\in A$ and $p\in P$, and find $q\in P$ and $\lambda\in\ell^1(I)$ such that
$p\le q$ and $p\le\lambda q$. Let
\[
I'=\{ i\in I : p_i>1\},\quad I''=\{ i\in I : p_i\le 1\}.
\]
Note that $\lambda_i>0$ whenever $i\in I'$.
For each $n\in\N$, define the subset $J'_n\subset I'$ as follows.
If $I'$ is finite, we set $J'_n=I'$. Otherwise, let
\[
J'_n=\{ i\in I' : q_i\le n\}.
\]
For each $i\in I'$ we have
\[
q_i\ge p_i/\lambda_i>1/\lambda_i\to\infty\qquad (i\in I',\; i\to\infty).
\]
Therefore $J'_n$ is finite.

Choose a finite subset $J''_n\subset I''$ such that
\begin{equation}
\label{1/nc_n}
\sup_{i\in I''\setminus J''_n} |a_i| p_i<1/n^2.
\end{equation}
Finally, let $J_n=J'_n\sqcup J''_n$ and $u_n=\chi_{J_n}$.
Since $P$ satisfies (\bfN), we have $\lambda(P)=\lambda^\infty(P)$ topologically
(see the discussion after Definition~\ref{def:N}),
and so the topology on $\lambda(P)$ is determined by the seminorms
$\|\cdot\|_p^\infty$ ($p\in P$). Therefore
by Lemma~\ref{lemma:a-au}, it suffices to show that
\[
\| a-au_n\|_p^\infty (1+\| u_n\|_p^\infty)\to 0 \qquad (n\to\infty).
\]
Since
\[
\| u_n\|^\infty_p=\sup_{j\in J_n} p_j \le n,
\]
we have
\[
\| a-au_n\|_p^\infty (1+\| u_n\|_p^\infty)
\le 2n\| a-au_n\|_p^\infty,
\]
and it remains to show that
\begin{equation}
\label{a-au_c_to_infty}
n\| a-au_n\|_p^\infty \to 0 \qquad (n\to\infty).
\end{equation}
Observe that
\begin{equation}
\label{a-au_c}
n\| a-au_n\|_p^\infty
=\sup_{i\in I\setminus J_n} |a_i|p_i n
=\max\Bigl\{ \sup_{i\in I'\setminus J'_n} |a_i|p_i n,
\sup_{i\in I''\setminus J''_n} |a_i|p_i n\Bigr\}.
\end{equation}
By \eqref{1/nc_n}, we have
\begin{equation}
\label{a-au_c''}
\sup_{i\in I''\setminus J''_n} |a_i|p_i n\le 1/n\to 0\qquad (n\to\infty).
\end{equation}
For each $i\in I'\setminus J'_n$ we have $q_i>n$. Therefore
\begin{equation}
\label{a-au_c'}
\sup_{i\in I'\setminus J'_n} |a_i|p_i n
\le \sup_{i\in I'\setminus J'_n} |a_i|q_i^2\to 0 \qquad (n\to\infty),
\end{equation}
because $P$ satisfies (\bfB). Now  \eqref{a-au_c_to_infty} follows
from \eqref{a-au_c}, \eqref{a-au_c''}, and \eqref{a-au_c'}.
\end{proof}

\begin{example}
Theorem~\ref{thm:BN_appr} implies that
$\Lambda_\infty(\alpha)$ is approximately contractible whenever
$\sup_n (\log n)/\alpha_n<\infty$
(see Examples~\ref{example:Lambda_R_nucl} and~\ref{example:Lambda_R_bipr}).
In particular, $s$ and $\cH(\CC)$ are approximately contractible.
\end{example}

\begin{example}
As was shown in \cite{DLZ}, $\ell^1$ is not approximately contractible
(although it satisfies (\bfB), see Example~\ref{example:l1_bipr}).
\end{example}

Let us now discuss some other properties of K\"othe algebras
related to approximate contractibility. One such property
is {\em idempotence}. Given an algebra $A$, let
\[
A^2=\spn\{ ab : a,b\in A\}.
\]
In \cite{DLZ}, the authors conjectured that there is a relation
between the approximate contractibility of a Banach sequence algebra $A$
and the property $A=A^2$. Below we study this relation in the case
of K\"othe algebras.

\begin{lemma}
\label{lemma:A=A^2}
Let $A=\lambda(P)$ be a K\"othe algebra. Then
\[
A^2=\{ a^2 : a\in A\}=\{ a\in A : \sqrt{|a|}\in A\}.
\]
\end{lemma}
\begin{proof}
Suppose that $\sqrt{|a|}\in A$. For each $i\in I$, find $\lambda_i\in\CC$
such that $a_i=\lambda_i^2 |a_i|$, and define $b_i=\lambda_i\sqrt{|a_i|}$.
We then have $b=(b_i)\in A$ and $a=b^2$. Clearly, this implies that $a\in A^2$.

Conversely, suppose that $a\in A^2$, and write
\[
a=\sum_{k=1}^n b_k c_k\quad (b_k,c_k\in A).
\]
For each $p\in P$ and each $i\in I$ we have
\[
|a_i|p_i^2
\le \sum_k |b_{ki}| |c_{ki}| p_i^2
\le \Bigl(\sum_k (|b_{ki}| p_i + |c_{ki}| p_i)\Bigr)^2.
\]
Therefore
\[
\sum_i \sqrt{|a_i|} p_i
\le\sum_i\sum_k (|b_{ki}|p_i + |c_{ki}|p_i)
=\sum_k (\| b_k\|_p+\| c_k\|_p)<\infty,
\]
which implies that $\sqrt{|a|}\in A$.
\end{proof}

\begin{remark}
An inspection of the above proof shows that, if $A$ is a K\"othe algebra
over $\R$, then
\[
A^2=\{ ab : a,b\in A\}=\{ a\in A : \sqrt{|a|}\in A\}.
\]
\end{remark}

Let $\lambda(P)$ be a K\"othe space. Recall that
the {\em K\"othe-Toeplitz dual} $\lambda(P)^\times$ is defined by
\[
\lambda(P)^\times=
\Bigl\{ y=(y_i)\in\CC^I :
\| y\|_x=\sum_i |y_i x_i| <\infty\quad\;\forall\, x\in\lambda(P)\Bigr\}.
\]
In other words, $\lambda(P)^\times$ is the K\"othe space $\lambda(P^\times)$,
where $P^\times=\{ (|x_i|): x\in\lambda(P)\}$.

\begin{lemma}
\label{lemma:nucl_idemp}
Let $A=\lambda(P)$ be a K\"othe algebra satisfying {\upshape (\bfB)}.
Suppose that $A^\times$ is nuclear. Then $A=A^2$.
\end{lemma}
\begin{proof}
Fix any $a\in A$.
Since $A^\times=\lambda(P^\times)$ is nuclear, the Grothendieck--Pietsch
criterion implies that $P^\times$ satisfies (\bfN).
Therefore there exist $\lambda\in\ell^1$ and $b\in P^\times$ such that $|a|\le\lambda b$.
For the same reason, there exist $\mu\in\ell^1$ and $c\in P^\times$ such that $b\le\mu c$.
Now take any $p\in P$ and find $q\in P$ and $C>0$ such that $p^2\le C^2q$.
We have
\begin{equation}
\label{sqrt_a}
\sum_i \sqrt{|a_i|} p_i
\le C\sum_i\sqrt{|a_i|q_i}
\le C\sum_i\sqrt{\lambda_i\mu_i c_i q_i}
\le C\sup_i\sqrt{ c_i q_i} \sum_i\sqrt{\lambda_i\mu_i}.
\end{equation}
Since $\lambda,\mu\in\ell^1$, we have $\sqrt{\lambda},\sqrt{\mu}\in\ell^2$,
and so $\sum_i\sqrt{\lambda_i\mu_i}<\infty$. Since $c\in A$, we have
$\sup_i\sqrt{c_i q_i}\le\sqrt{\| c\|_q}<\infty$.
Now \eqref{sqrt_a} implies that $\sum_i \sqrt{|a_i|} p_i<\infty$, i.e., $\sqrt{|a|}\in A$.
Using Lemma~\ref{lemma:A=A^2}, we see that $a\in A^2$, and so $A=A^2$, as claimed.
\end{proof}

\begin{lemma}
\label{lemma:log_nucl}
Let $A=\lambda(P)$ be a K\"othe algebra on $\N$ satisfying {\upshape (\bfB)}.
Suppose that $p_n\ge 1$ for all
$p\in P$ and all $n\in\N$, and suppose that there exists $p\in P$ such that
\begin{equation}
\label{log}
\sup_n \frac{\log n}{\log p_n}<\infty.
\end{equation}
Then $A$ is nuclear.
\end{lemma}
\begin{proof}
Since $P$ is directed (Axiom (P2)), the topology on $\lambda(P)$ is generated by all
seminorms $\|\cdot\|_q$ with $q\in P,\; q\ge p$. Therefore me may assume that
\eqref{log} holds for all $p\in P$.
Now fix any $p\in P$ and find $k\in\N$ such that $\log n\le k\log p_n$ for all $n\in\N$.
Since $P$ satisfies (\bfB), there exist $C>0$ and $q\in P$ such that
$p^{2k+1}\le Cq$. We have
\[
\sum_n \frac{p_n}{q_n}
\le C\sum_n \frac{1}{p_n^{2k}}
\le C\sum_n \frac{1}{n^2}<\infty.
\]
Now the Grothendieck--Pietsch criterion implies that $A$ is nuclear.
\end{proof}

\begin{lemma}
\label{lemma:idemp_log}
Let $A=\lambda(P)$ be a K\"othe algebra on $\N$ such that $A=A^2$.
Suppose that $1\le p_n\le p_{n+1}$ for all $p\in P$ and all $n\in\N$.
Then there exists $p\in P$ such that \eqref{log} holds.
\end{lemma}
\begin{proof}
Without loss of generality, we may assume that
$P=\{ p^{(n)} : n\in\N\}$, and that $p^{(n)}\le p^{(n+1)}$ for all $n$.
Assume, towards a contradiction, that none of $p\in P$ satisfies \eqref{log}.
Then for each $n\in\N$ there exists $k_n\in\N$ such that
\[
\frac{\log k_n}{\log p_{k_n}^{(n)}}\ge n, \quad\text{i.e., }
p_{k_n}^{(n)}\le k_n^{1/n}.
\]
We may also assume that $k_{n+1}\ge 2k_n$ for all $n\in\N$. Set $k_0=0$
and define $a\in\CC^\N$ by
\[
a_m=\frac{1}{k_n^3}\quad\text{for } k_{n-1}<m\le k_n.
\]
Then for each $i\in\N$ we have
\[
\begin{split}
\sum_{m>k_{i-1}} |a_m| p_m^{(i)}
&=\sum_{n\ge i} \;\sum_{m=k_{n-1}+1}^{k_n} |a_m|p_m^{(i)}
\le \sum_{n\ge i}\; \sum_{m=k_{n-1}+1}^{k_n} \frac{1}{k_n^3} p_{k_n}^{(n)}\\
&\le \sum_{n\ge i}\; \sum_{m=k_{n-1}+1}^{k_n} k_n^{\frac{1}{n}-3}
= \sum_{n\ge i} (k_n-k_{n-1}) k_n^{\frac{1}{n}-3}
\le \sum_{n\ge i} k_n^{\frac{1}{n}-2}<\infty,
\end{split}
\]
because $\frac{1}{n}-2<-\frac{3}{2}$ for $n\ge 2$. Therefore $a\in A$.
On the other hand,
\[
\sum_m \sqrt[4]{|a_m|}
=\sum_n \;\sum_{m=k_{n-1}+1}^{k_n}\; \frac{1}{k_n^{3/4}}
\ge\sum_n \frac{k_n-k_{n-1}}{k_n}
\ge\sum_n \frac{1}{2}=\infty,
\]
showing that $\sqrt[4]{|a|}\notin\ell^1$, and, {\em a fortiori},
$\sqrt[4]{|a|}\notin A$. Applying Lemma~\ref{lemma:A=A^2} twice,
we see that $a\notin A$. The resulting contradiction completes the proof.
\end{proof}

Now we can summarize the above discussion as follows.

\begin{theorem}
\label{thm:idemp_appr}
Let $A=\lambda(P)$ be a K\"othe algebra on $\N$ satisfying {\upshape (\bfB)}.
Suppose that $1\le p_n\le p_{n+1}$ for all $p\in P$ and all $n\in\N$.
Then the following conditions are equivalent:
\begin{compactenum}
\item[{\upshape (i)}] $A$ is nuclear;
\item[{\upshape (ii)}] $A^\times$ is nuclear;
\item[{\upshape (iii)}] $A=A^2$;
\item[{\upshape (iv)}] there exists $p\in P$ such that
\[
\sup_n \frac{\log n}{\log p_n}<\infty.
\]
\end{compactenum}
Each of the above conditions implies that $A$ is approximately contractible.
\end{theorem}
\begin{proof}
$\mathrm{(i)}\Longrightarrow\mathrm{(ii)}$.
This is true for all metrizable K\"othe spaces \cite[Satz 7]{Kothe_nucl}.

$\mathrm{(ii)}\Longrightarrow\mathrm{(iii)}$. This follows from
Lemma~\ref{lemma:nucl_idemp}.

$\mathrm{(iii)}\Longrightarrow\mathrm{(iv)}$. This follows from
Lemma~\ref{lemma:idemp_log}.

$\mathrm{(iv)}\Longrightarrow\mathrm{(i)}$. This follows from
Lemma~\ref{lemma:log_nucl}.

Finally, (i) implies that $A$ is approximately contractible by Theorem \ref{thm:BN_appr}.
\end{proof}

\section{Open problems}
In conclusion, let us formulate some open problems.

\begin{problem}
\label{probl:bifl_wdb}
Let $A$ be a biflat Fr\'echet algebra. Is $\wdb A\le 2$?
\end{problem}

Recall that the answer is positive if $A$ is a Banach algebra
\cite[2.5.8]{X_HOA} or if $A$ is projective in $A\lmod$ (or in $\rmod A$)
\cite[Proposition 4.8]{Pir_msb}.

In fact, Problem~\ref{probl:bifl_wdb} can be reduced to the following.

\begin{problem}
Let $X$ be flat in $A\lmod$, and let
$Y$ be flat in $\rmod A$. Is $X\Ptens Y$ flat in $A\bimod A$?
\end{problem}

Again, the answer is positive in the case of Banach modules over Banach algebras
\cite[7.1.57]{X2}. The answer is also positive provided that either $X$ or $Y$
is projective \cite[Proposition 3.6]{Pir_msb}.

\begin{problem}
Let $A=\lambda(P)$ be an approximately contractible K\"othe algebra.
\begin{compactenum}
\item[(1)] Does $P$ satisfy (\bfB)?
\item[(2)] Does $A$ satisfy the conditions (i)--(iv) of Theorem \ref{thm:idemp_appr}?
In particular, must $A$ be nuclear?
\end{compactenum}
\end{problem}

Note that, if $A=A^2$, then $P$ satisfies (\bfB)
(see implication $\mathrm{(v)}\Longrightarrow\mathrm{(i)}$ in Theorem~\ref{thm:kothe_bipr}).
Therefore a positive answer to question (2) would imply a positive answer to question~(1).

\section*{Acknowledgments}
This paper is based on a lecture delivered at the 19$^\mathrm{th}$
International Conference on Banach Algebras held at
{\fontencoding{T1}\selectfont B\k{e}dlewo}, July 14--24,
2009. The support for the meeting by the Polish Academy of Sciences, the
European Science Foundation under the ESF-EMS-ERCOM partnership, and the
Faculty of Mathematics and Computer Science of the Adam Mickiewicz University
at Pozna\'n is gratefully acknowledged.
The author was also supported by the RFBR grant 08-01-00867,
by the Ministry of Education and Science of Russia (programme ``Development of the scientific
potential of the Higher School'', grant no. 2.1.1/2775),
and by the President of Russia grant MK-1173.2009.1.

The author is grateful to C.~J.~Read for explaining some points of \cite{LR}.

\end{document}